\documentclass[11pt]{article}

\usepackage{graphicx} 
\usepackage{amssymb, amsmath, amsthm, amsfonts}

\usepackage[dvipsnames]{xcolor}

\usepackage{tikz}
\usepackage{stmaryrd} 
\usepackage{thm-restate}
\usepackage{mathtools}
\usepackage{algpseudocode}
\usepackage{algorithm}
\usepackage{enumerate}
\usepackage{comment}
\usepackage[margin=1in]{geometry}

\usepackage{authblk}

\usepackage{hyperref}
\hypersetup{
  linkcolor  = Bittersweet,
  citecolor  = MidnightBlue,
  colorlinks = true
}
\usepackage{cleveref}

\newtheorem{theorem}{Theorem}[section]
\newtheorem{lemma}[theorem]{Lemma}
\newtheorem{cor}[theorem]{Corollary}
\newtheorem{conj}[theorem]{Conjecture}
\newtheorem{question}[theorem]{Question}
\newtheorem{prop}[theorem]{Proposition}

\newtheorem{claim}{Claim}

\theoremstyle{definition}
\newtheorem{definition}[theorem]{Definition}
\newtheorem{remark}[theorem]{Remark}
\newtheorem{example}{Example}
\DeclareMathOperator\dist{dist}
\newcommand{\N}{\mathbb{N}}

\title{Small $q$-kernels in digraphs with minimum in-degree $\delta$}

\author[1]{Geoffrey Boyer}
\author[2]{Matt Burnham\thanks{mattmath@iastate.edu}}
\author[3]{Daniela \v{C}ern\'a\thanks{daniela.cerna@cvut.cz}}
\author[4]{Stephen G. Hartke\thanks{stephen.hartke@ucdenver.edu}}
\author[5]{Isaiah Hollars\thanks{isaiah.hollars@sc.edu}}
\author[2]{Joel Jeffries}
\author[2]{Sydney Miyasaki}
\author[6]{Tobias Timofeyev}

\affil[1]{College of Charleston}
\affil[2]{Iowa State University}
\affil[3]{Czech Technical University in Prague} 
\affil[4]{University of Colorado Denver}
\affil[5]{University of South Carolina}
\affil[6]{University of Vermont}

\begin{document}

\maketitle

\begin{abstract}
For a digraph $D$, a subset $Q\subseteq V(D)$ is called a $q$-kernel if $Q$ is an independent set and all vertices in $V(D)$ are reachable from $Q$ via a directed path of length at most $q$. Given integers $q\geq 2$ and $\delta\geq 1$, Spiro~\cite{spiro2026generalized} posed the question: what is the smallest constant $c_{\delta,q}$ such that every digraph $D$ with minimum in-degree $\delta$ has a $q$-kernel of size at most $c_{\delta,q}|V(D)|$? We show the constants $c_{\delta,q}$ are monotone in both $\delta$ and $q$, and we improve upon the known upper bounds for $c_{\delta,q}$. Our main results show $\frac{1}{\delta+1} \leq c_{\delta,q}\leq \frac{1}{\lfloor\sqrt{\delta+1}\rfloor+1}$
for all $q \geq 3$ and $\delta \geq 1$, and $ c_{\delta,q}=\frac{1}{\delta+1}$ whenever $\delta \geq 1$ and $q \geq \left\lceil\frac{3\delta}{2}\right\rceil + 1$.

\end{abstract}

\newpage 

\section{Introduction}
\subsection{Background}

Throughout this paper, all digraphs are finite and contain no loops or multi-edges. See Section~\ref{sec-prelim} for basic digraph terminology. Let $D$ be a digraph. A set $Q \subseteq V(D)$ is called a \textit{kernel} if it is an independent set and all vertices in $V(D)$ are contained in $Q$ or the out-neighborhood of $Q$.\footnote{We note that some authors define a kernel with the opposite orientation. While Erd\H{o}s et al.~\cite{EGMST} define kernels with outwardly oriented edges as we do, Von Neumann and Morgenstern~\cite{NeumannMorgenstern} defined kernels with inwardly oriented edges instead.}

Not every digraph contains a kernel. For example, a directed odd cycle contains no kernel. For this reason, one might consider weakening the connectivity requirement.
An independent set $Q\subseteq V(D)$ is a \textit{quasikernel} if all vertices in $V(D)$ are reachable from $Q$ via a directed path of length at most $2$. Landau~\cite{Landau_1953} showed that every tournament has a quasikernel that contains a single vertex, often referred to as a king. Later, Chvátal and Lovász~\cite{ChvatalLovasz} proved that every digraph has a quasikernel. The natural question which follows is how small in relation to $|V(D)|$ a quasikernel can be.

\begin{conj}{(Small Quasikernel Conjecture~\cite{katona2011fete})\footnote{The small quasikernel conjecture was made in 1976 by P.L. Erd\H{o}s and L.A. Sz\'ekely (long before it appeared in the literature).}}\label{conj-smallquasi}
    Every source-free digraph $D$ contains a quasikernel with at most $\frac{|V(D)|}{2}$ vertices. 
\end{conj}
Notice that the source-free condition is necessary (e.g., consider an edgeless digraph or a star with $n-1$ sources and $1$ sink).
The small quasikernel conjecture has been well studied, but its validity is still an open question. We refer the reader to the survey paper by P. L. Erd{\H{o}}s et al.~\cite{EGMST} for more historical background and progress on the conjecture. See also the more recent works \cite{ai2024variable,clow2026greedily}. 

In 2024, Spiro~\cite{spiro2026generalized} generalized the definition of a quasikernel by relaxing the connectivity condition. A subset $Q\subseteq V(D)$ is called a \textit{$q$-kernel} if $Q$ is an independent set and all vertices in $V(D)$ are reachable from $Q$ via a directed path of length at most $q$. In particular, quasikernels are $2$-kernels. Spiro posed several asymptotic questions regarding the size of a smallest $q$-kernel analogous to the small quasikernel conjecture. Similar concepts have previously been investigated in graphs. See \cite{davila2015lower,hansberg2007distance,henning2017distance} for bounds on $k$-dominating sets and \cite{bujtas2025revisiting} for results on independent dominating sets for bipartite graphs and trees. One approach to study these questions has been to restrict attention to strongly connected digraphs, which are well behaved and naturally source-free. See Spiro~\cite{spiro2026generalized} Question 7.7 and the paper of Nguyen, Scott, and Seymour~\cite{nguyen2024distant} for progress in this direction.

Note that being source-free is equivalent to having minimum in-degree at least $1$. More generally, vertices with small minimum in-degree seem to be major obstacles for creating small $q$-kernels. This motivated Spiro to define the following constants.

\begin{definition}
	Define $c_{\delta,q}$ to be the smallest constant such that every digraph $D$ with minimum in-degree $\delta$ contains a $q$-kernel of size at most $c_{\delta,q}|V(D)|$. 
\end{definition}

The Small Quasikernel Conjecture (in conjunction with Theorem~\ref{thm-mono}) is equivalent to the statement that $c_{1,2}\leq \frac{1}{2}$. The goal of this paper is to improve upon the known bounds for the constants $c_{\delta,q}$. The previous best bounds are summarized by Spiro (Proposition 7.5 in \cite{spiro2026generalized}):
\begin{prop}[Spiro~\cite{spiro2026generalized}]\label{prop-spiro-cdeltaq}
	\phantom{x}
\begin{enumerate}[(a)]
	\item For all $\delta \geq 1$ and $q \geq 2$, $c_{\delta, q}\geq \frac{1}{\delta+1}$.
	\item If $q \geq 2^{\delta+2}$, then $c_{\delta, q}= \frac{1}{\delta+1}$.
	\item For all $\delta \geq 1$, we have $c_{\delta,2} \geq \frac{1}{2}$.
\end{enumerate}
\end{prop}
\begin{example}\label{ex-deltaplusone}
    Let $\delta\geq 1$ and let $\overleftrightarrow{K}_{\delta+1}$ denote the bidirected complete graph on $\delta+1$ vertices. The minimum in-degree of $\overleftrightarrow{K}_{\delta+1}$ is $\delta$ and the size of a smallest $q$-kernel is 1 for all $q\geq 2$. Thus, $c_{\delta, q}\geq \frac{1}{\delta+1}$ for all $\delta\geq 1$ and $q\geq 2$. 
\end{example}

\subsection{Results}
We now discuss our main results, proved in later sections. First, we show that the constants $c_{\delta, q}$ are monotone in $\delta$.

\begin{restatable*}{theorem}{restatemono}
    \label{thm-mono}
	Let $q\geq 2$, $\delta\geq1$ be fixed integers. Then $c_{\delta+1,q}\leq c_{\delta,q}$.
\end{restatable*}

While this result should be intuitive, its justification requires some care, shown in Section \ref{section-monotonicity}. Increasing minimum in-degree increases the connectivity, making it easier for a possible $q$-kernel to reach more vertices. However, increasing the minimum in-degree also makes it harder for such a set to meet the independence condition of $q$-kernels. 
On the other hand, $c_{\delta,q+1}\leq c_{\delta,q}$ trivially holds since every $q$-kernel is also a $(q+1)$-kernel.

Consider the case where $q = 2$. Proposition~\ref{prop-spiro-cdeltaq}~(c) gives a lower bound of $c_{\delta,2}\geq \frac{1}{2}$. The Small Quasikernel Conjecture is that equality holds. Surprisingly, it is not even known if there exists some $\delta$ such that $c_{\delta,2}<1$. Our next result shows the behavior of $c_{\delta,q}$ for any fixed $q\geq 3$ is completely different. If $q\geq 3$, then $c_{\delta,q}\to 0$ as $\delta\to\infty$.

\begin{restatable*}{theorem}{thmGeneralq}
\label{thm-general q=3}
Let $\delta \geq 1$. Every digraph $D$ with minimum in-degree $\delta$ contains a $3$-kernel of size at most $\frac{|V(D)|}{\lfloor\sqrt{\delta+1}\rfloor +1}$. 
In other words, $c_{\delta,3}\leq \frac{1}{\lfloor\sqrt{\delta+1}\rfloor +1}$. This immediately implies $c_{\delta,q}\leq \frac{1}{\lfloor\sqrt{\delta+1}\rfloor +1}$ for all $q\geq3$.
\end{restatable*}

Now, suppose we fix the constant $\delta$. Then the constants $c_{\delta_,q}$ are weakly monotone-decreasing in $q$. That is, $c_{\delta,q+1}\leq c_{\delta,q}$. By Proposition~\ref{prop-spiro-cdeltaq} (a) and (b), there exists some threshold value $q'$ such that for all $q\geq q'$, we have $c_{\delta,q}=\frac{1}{\delta+1}$. In particular, Proposition \ref{prop-spiro-cdeltaq} (b) says we can take $q'=2^{\delta+2}$. We show $q'\approx \frac{3}{2}\delta$ suffices. 

\begin{restatable*}{theorem}{thmGeneraldelta}
    \label{thm-1/(1+delta)}
	Let $\delta \ge 1$. If $D$ is a digraph with minimum in-degree $\delta$, then $D$ has a $\left( \lceil \frac{3}{2}\delta \rceil +1 \right)$-kernel of size at most $\frac{1}{\delta+1}|V(D)|$.
	Hence for any $q \geq \lceil \frac{3}{2}\delta \rceil +1 $, $c_{\delta, q}= \frac{1}{\delta+1}$. 
\end{restatable*}

\section{Preliminaries and Organization}\label{sec-prelim}

All of our notation (except perhaps the $d$-neighborhood notation) is standard. A simple digraph $D$ consists of a finite set of vertices $V(D)$ and a set of arcs $E(D)\subseteq (V(D)\times V(D))\smallsetminus \{(u,u): u\in V(D)\}$. All digraphs in this paper are simple digraphs. A vertex $u$ is an \textit{in-neighbor} of $v$ if $(u,v)\in E(D)$. We use $N^-(v)$ to denote the set of in-neighbors of $v$ and let $d^-(v)= |N^-(v)|$ denote the \textit{in-degree} of $v$. Analogously, $u$ is an \textit{out-neighbor} of $v$ if $(v,u)\in E(D)$, $N^+(v)$ denotes the set of out-neighbors of $v$, and $d^+(v)= |N^+(v)|$. For $A\subseteq V(D)$, $N^-_A(v)\coloneqq N^-(v)\cap A$ and $d^-_A(v)\coloneqq |N^-_A(v)|$. The \emph{minimum in-degree} of a digraph $D$ is given by $\min_{v\in V(D)}d^-(v)$. A vertex $v$ is a \textit{source} if it has no in-neighbors, i.e. $d^-(v) = 0$.

A \textit{directed path} of length $\ell$ (from $v_0$ to $v_\ell$) is a sequence of distinct vertices $v_0v_1\dots v_\ell$ such that $(v_i,v_{i+1})\in E(D)$ for $0\leq i\leq \ell-1$. Similarly, a \textit{directed cycle} of length $\ell\geq2$ is a sequence of distinct vertices $v_0v_1\dots v_{\ell-1} v_0$ such that $(v_i,v_{i+1})\in E(D)$ for $0\leq i\leq \ell-1$ with addition modulo $\ell$. We say $D$ is \textit{acyclic} if it contains no directed cycles. A useful fact is that $D$ is acyclic if and only if its vertices have a linear ordering, i.e., an ordering $v_1,\dots,v_n$ such that all arcs $(v_i,v_j)\in E(D)$ satisfy $i<j$.

Let $A\subseteq V(D)$. For $u,v\in A$, we define $\dist_A(u,v)$ to be the smallest integer $\ell\in \{0,1,\dots\}$ such that there exists a directed path of length $\ell$ from $u$ to $v$ with all its vertices in $A$. If no such path exists, we set $\dist_A(u,v)=+\infty$. 
For $X \subset A$ and $v \in A$, we define $\dist_A(X,v) = \min\{\dist_A(u,v): u \in X\}$.
For $d \geq 1$ and $v \in A$, we define the \textit{closed $d$-neighborhood} of $v$ as \[N^{\leq d}_A[v] = \{w\in A : \dist_A(v,w) \leq d\}.\] The \textit{open $d$-neighborhood} is defined as $N^{\leq d}_A(v) = N^{\leq d}_A[v]\smallsetminus \{v\}$. For a subset $S \subseteq A$, define $N_{A}^{\le d}[S] = \bigcup_{v \in S} N_{A}^{\le d}[v]$ and $N_{A}^{\le d}(S) = N_{A}^{\le d}[S]\smallsetminus S$. When $A=V(D)$, we omit the subscript $A$ and just write $\dist(u,v)$, $N^{\leq d}[v]$, and $N^{\le d}[S]$. For $A\subseteq V(D)$, the induced digraph $D[A]$ has vertex set $A$ and arc set $\{(v,u)\in E(D): v,u\in A\}$. We say $A$ is an \textit{independent set} if $D[A]$ contains no arcs. For completeness, we restate the main definitions from the introduction.
\begin{definition}[\cite{spiro2026generalized}]
A set $Q\subseteq V(D)$ is said to be a $q$-\textit{kernel} of $D$ if $Q$ is an independent set and $N^{\le q}[Q] = V(D)$. In the literature, a $1$-kernel is called a \emph{kernel}, and a $2$-kernel is called a \emph{quasikernel}. For positive integers $\delta \geq 1$ and $q\geq 2$, the constant $c_{\delta,q}>0$ is defined to be the smallest constant such that every digraph $D$ with minimum in-degree $\delta$ contains a $q$-kernel of size at most $c_{\delta,q}|V(D)|$. Explicitly, 
    \begin{align*}
    c_{\delta,q}\coloneqq \inf\bigg\{c>0:  &\text{ all digraphs $D$ with minimum in-degree $\delta$}\\
    &\text{contain a $q$-kernel of size at most $c|V(D)|$}\bigg\}. 
\end{align*}
\end{definition}

The rest of the paper is structured as follows. The proofs of Theorem~\ref{thm-mono}, Theorem~\ref{thm-general q=3}, and Theorem~\ref{thm-1/(1+delta)} are given in Sections \ref{section-monotonicity}, \ref{section-general q=3}, and \ref{section-1/(1+delta)} respectively. Section \ref{section-algorithm} introduces an algorithm we use extensively in the proofs of the theorems. We conclude with some open questions in Section \ref{section-openquestions}.

\section{Monotonicity of $c_{\delta,q}$}\label{section-monotonicity}
In this section, we prove the following.
\restatemono

We begin by introducing an asymptotic variation of this problem. Namely, let $\tilde{c}_{\delta,q}$ be the smallest constant for which every \textit{sufficiently large} digraph $D$ with minimum in-degree $\delta$ contains a $q$-kernel $Q$ with $|Q| \leq \tilde{c}_{\delta,q} |V(D)|$. In more detail,
\begin{align*}
    \tilde{c}_{\delta,q}\coloneqq \inf\bigg\{c>0: \exists N=N(c) \text{ such that all digraphs $D$ on at least $N$ vertices}\\
    \text{with min in-degree $\delta$ have a $q$-kernel of size at most $c|V(D)|$}\bigg\}. 
\end{align*}
This variation is sometimes easier to work with and, as the next lemma shows, is equivalent.

\begin{lemma}[asymptotic equivalence]\label{lem:asymptotic_equivalence}
Let $q \geq 2$, $\delta \geq 1$ be fixed integers. For any $\varepsilon > 0$, there exists a sequence of digraphs $\{D_k\}_{k \in \N}$ such that $\{ |V(D_k)| \}_{k \in \N}$ is increasing, the minimum in-degree of $D_k$ is $\delta$, and the size of the smallest $q$-kernel in $D_k$ is at least $(c_{\delta, q} - \varepsilon) |V(D_k)|$. Hence $\tilde{c}_{\delta,q}\geq c_{\delta,q}$. Since $\tilde{c}_{\delta,q}\leq c_{\delta,q} $ holds by definition, we deduce \[c_{\delta,q} = \tilde{c}_{\delta,q}.\]
\end{lemma}
\begin{proof}
Fix $\varepsilon > 0$. If, for all $n$ and all digraphs $D$ on $n$ vertices with minimum in-degree $\delta$, there exists a $q$-kernel of size less than $(c_{\delta,q} - \varepsilon) n$, then we have a contradiction with the definition of $c_{\delta, q}$. So, there is some $n$ and some digraph $D$ on $n$ vertices with minimum in-degree $\delta$ whose smallest $q$-kernel has size at least $(c_{\delta,q} - \varepsilon)n$. Now, for each $k \in \N$, let $D_{k}$ be the digraph formed by taking $k$ disjoint copies of $D$, say $H_1, \dots, H_k$. Note that the minimum in-degree of $D_k$ remains at $\delta$. Consider a $q$-kernel $Q_k$ in $D_k$. Every vertex in $D_k$ can be reached from $Q_k$ by a directed path of length $q$. But since $H_i$ and $H_j$ are disjoint for all $i \neq j$, it must be that every vertex in $H_i$ is reached from $Q_k \cap H_i$ by a path contained in $H_i$. That is, each of the sets $Q_k \cap H_i$ is a $q$-kernel of $H_i$. Hence, $Q_k \cap H_i$ contains at least $(c_{\delta,q} - \varepsilon) n$ vertices, and so $|Q_k| \geq (c_{\delta,q} - \varepsilon) nk = (c_{\delta,q} - \varepsilon) |V(D_k)|$. The sequence $\{D_k\}_{k\in \N}$ has the desired property.
\end{proof}

\begin{lemma}[sink trick lemma]\label{lem:sinktricklemma}
    Let $q\geq 2$ and $\delta\geq 1$ be positive integers and let $D$ be a digraph on $n$ vertices with minimum in-degree $\delta+1$. Then $D$ contains a $q$-kernel $Q$ with \[
    |Q|\leq \left(c_{\delta,q}+\frac{1}{n}\right)n.
    \]
\end{lemma}
\begin{proof}
Let $D$ be a digraph on $n$ vertices with minimum in-degree $\delta + 1$. Add to $D$ a sink $u$ with in-degree exactly $\delta$, and let $D'$ denote the resulting digraph (the in-neighbors of $u$ can be chosen arbitrarily). Then $D'$ has minimum in-degree $\delta$ and so contains a $q$-kernel $Q'$ with \[
   |Q'|\leq c_{\delta,q} (n + 1) = c_{\delta,q} n + c_{\delta,q} \leq c_{\delta,q} n + 1 = \left(c_{\delta,q} + \frac{1}{n}\right)n.
    \] We claim that $Q'\smallsetminus\{u\}$ is a $q$-kernel for $D$. Clearly $Q'\smallsetminus\{u\}$ is independent in $D$. Let $v\in V(D)\smallsetminus (Q'\smallsetminus\{u\})= V(D)\smallsetminus Q'$. Then there exists a directed path in $D'$ from $Q'$ to $v$ of length at most $q$. Since $u$ is a sink, $u$ is not on this path. Consequently, there is a directed path in $D$ of length at most $q$ from $Q'\smallsetminus\{u\}$ to $v$. Hence $Q' \smallsetminus \{u\}$ is a $q$-kernel for $D$ of size at most $\left(c_{\delta,q} + \frac{1}{n}\right)n$. 
\end{proof}
We now prove Theorem~\ref{thm-mono}. 
\begin{proof}
    Let $q \ge 2$ and $\delta \ge 1$. Let $N \ge 1$ be arbitrarily large, and $\varepsilon > 0$ arbitrarily small. By \Cref{lem:asymptotic_equivalence}, there exist $n \ge N$ and a digraph $D$ on $n$ vertices with minimum in-degree $\delta+1$ whose smallest $q$-kernel has size at least $(c_{\delta+1,q}-\varepsilon)n$. By \Cref{lem:sinktricklemma}, $D$ has a $q$-kernel of size at most $(c_{\delta,q}+\frac{1}{n})n$. Therefore,
\[
    (c_{\delta+1,q}-\varepsilon)n \le \left(c_{\delta,q}+\frac{1}{n}\right)n \implies c_{\delta+1,q}-\varepsilon \le c_{\delta,q}+\frac{1}{N}.
\]
Since $\frac1N$ and $\varepsilon$ are arbitrarily small, we obtain $c_{\delta+1,q} \le c_{\delta,q}$.
\end{proof}

\begin{remark}
The following definition is also quite natural. Define
\begin{align*}
    c_{\delta,q}'\coloneqq \inf\bigg\{c>0:  &\text{ all digraphs $D$ with minimum in-degree at least $\delta$}\\
    &\text{contain a $q$-kernel of size at most $c|V(D)|$}\bigg\}. 
\end{align*}
Note that $c'_{\delta,q}$ only differs from $c_{\delta,q}$ by the inclusion of ``at least'' in the definition. Clearly $c_{\delta,q}\leq c'_{\delta,q}$. Theorem~\ref{thm-mono}  implies $c'_{\delta,q} \leq c_{\delta,q}$ for all $q\geq 2, \delta\geq 1$. So in fact, $c_{\delta,q} = c'_{\delta,q} = \tilde{c}_{\delta,q}$ for all $q\geq 2, \delta\geq 1$. 
\end{remark}

\section{The modified first-phase Chvátal--Lovász algorithm}\label{section-algorithm}

Instead of constructing $q$-kernels directly, our proofs often focus on constructing the following object which tends to be easier to work with. 

\begin{definition}
	A set $R\subseteq V(D)$ is said to be a $q$-\textit{prekernel} if $D[R]$ is acyclic and $N^{\le q}[R] = V(D)$.
\end{definition}
This is almost the definition of a $q$-kernel except that we only require the set to be acyclic instead of independent. A $q$-prekernel naturally gives rise to a $(q+1)$-kernel as we show in the following lemma.

\begin{lemma}\label{lem:prekernel}
	If a digraph $D$ contains a $q$-prekernel $R$, then it contains a $(q + 1)$-kernel $Q$ with $Q\subseteq R$. Furthermore, if $R$ induces an arc, then $Q$ is a proper subset of $R$ and $|Q| < |R|$.
\end{lemma}

\begin{proof}
    Let $R$ be a $q$-prekernel of $D$. Since $R$ is acyclic, there exists an ordering $v_1,v_2,\dots,v_{|R|}$ of the vertices of $R$ such that all arcs $(v_i,v_j)$ induced by $R$ satisfy $i< j$.  Initialize $Q=\emptyset$ and $A=R$. While $A$ is nonempty, let $i=\min\{j:v_j\in A\}$. Add $v_i$ to $Q$ and remove $N^{\leq 1}[v_i]$ from $A$. Iterate this procedure until $A$ is empty. We claim that the final set $Q$ is independent. Suppose not. Then there is some arc $(v_i,v_j)$ in $Q$ with $i < j$. By the construction of $Q$, $v_i$ was added before $v_j$. But when $v_i$ was added to $Q$, $v_j\in N^{\leq 1}[v_i]$ was removed from $A$ and could not have been chosen in any later step, a contradiction. For $u\in V(D)$, there exists a (shortest) directed path of length at most $q$ from some $v_i\in R$ to $u$. By construction, either $v_i\in Q$ or $v_i$ has an in-neighbor in $Q$. In either case, we obtain a directed path of length at most $q+1$ from $Q$ to $u$ (the in-neighbor cannot lie on the path from $v_i$ to $u$ since we chose the shortest path). Hence, $Q$ is a $(q+1)$-kernel. Notice that if $R$ contains at least one arc, then $|Q|<|R|$ due to the removal of $N^{\leq 1}[v_j]$ for each $v_j\in Q$.
\end{proof}

\begin{remark}\label{remark-1pre}
    The same technique can be used to show every digraph $D$ contains a $1$-prekernel. Order the vertices of $V(D)$ arbitrarily. Initialize $R=\emptyset$ and $A=V(D)$. While $A$ is nonempty, let $v_i\in A$ with $i=\min\{j:v_j\in A\}$. Add $v_i$ to $R$ and remove $N^{\leq 1}[v_i]$ from $A$. Iterate this procedure until $A$ is empty. The resulting set $R$ is a 1-prekernel for $D$. This is exactly the first phase of what the authors of \cite{EGMST} refer to as the Chvátal--Lovász algorithm (the algorithm is a reformulation of the original induction proof of Chvátal and Lovász~\cite{ChvatalLovasz} that every digraph has a quasikernel).
\end{remark}

Algorithm~\ref{alg:RAB} below is similar to the first phase of the Chvátal--Lovász algorithm with the major difference being the added flexibility of parameters $\ell$ and $k$. It is important to note that the algorithm does not directly output $q$-kernels or $q$-prekernels. The algorithm  outputs a partition $V(D)=R\sqcup A\sqcup B$ with several nice properties described by the following lemma. In particular, the set $R$ will be acyclic and will serve as our starting point for constructing small $q$-prekernels.

\begin{algorithm}
\caption{An algorithm for generating sets $R,A,B$}\label{alg:RAB}
\begin{algorithmic}
\Require $\ell \geq 1$, $k \geq 1$, digraph $D$
\State $R \gets \emptyset$ \Comment{$R$ is the starting point for our $q$-prekernel.}
\State $B \gets \emptyset$ \Comment{$B$ is the set of vertices ``covered'' by $R$.}
\State $A \gets V(D)$ \Comment{$A$ is the set of vertices not ``covered'' by $R$.}
\While{$\exists v \in A$ such that $N_{A}^{\leq \ell} (v) \geq k$}
\State $R \gets R \cup \{v\}$                         
\State $B \gets B \cup (N^{\leq \ell} (v)\smallsetminus R) $  \Comment{We add at least $k$ vertices to $B$.}
\State $A \gets A \smallsetminus N^{\leq \ell} [v]$  \Comment{We delete at least $k+1$ vertices from $A$.}
\EndWhile
\end{algorithmic}
\end{algorithm}

\begin{figure}[h]
    \centering
    \includegraphics[width=0.6\linewidth]{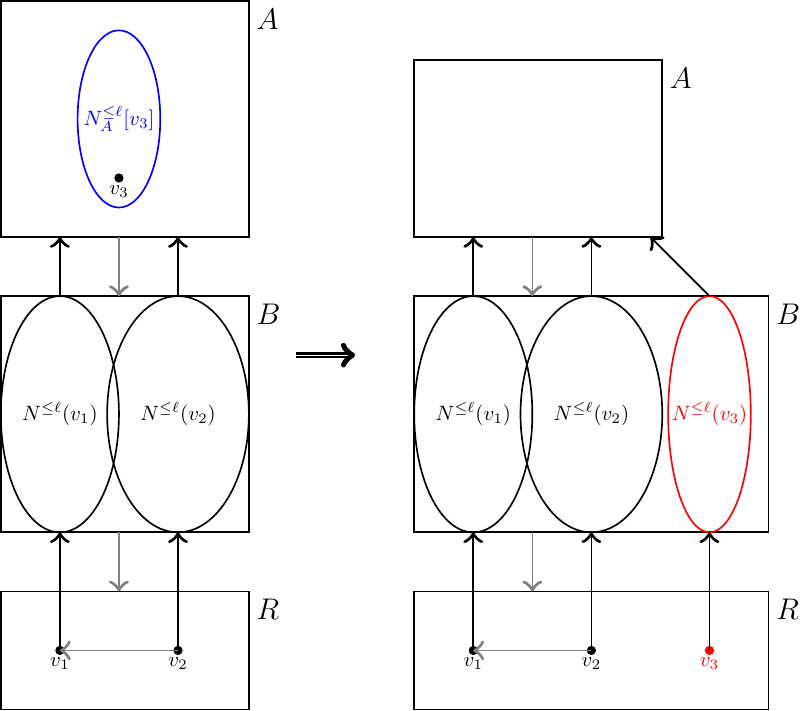}
    \caption{Algorithm 1 moving $v_3$ to $R$ and moving $N_A^{\leq \ell}(v_3)$ to $B$.}
    \label{fig:alg1}
\end{figure}

\begin{lemma}\label{lem:ALG}
    Let $D$ be a digraph and $\ell\geq 1$ and $k\geq 1$ be positive integers. Then the output sets $R,A,B$ of Algorithm~\ref{alg:RAB} satisfy the following.
    \begin{enumerate}[(1)]
        \item $R\sqcup A\sqcup B$ is a partition of $V(D)$.
        \item $R$ is acyclic.
        \item $k |R| \leq |B|$.
        \item For every $a \in A$, $|N_{A}^{\leq \ell}(a)| < k$.
        \item For every $b \in B$, $\dist(R,b) \leq \ell$.
        \item Suppose $\ell\geq k$ and that every vertex in $V(D)$ is an endpoint of a directed path of length $k$. Then for every $a\in A$, we have $\dist(B,a) \leq k$ and $\dist(R,a) \leq \ell+k$.
    \end{enumerate}
\end{lemma}
\begin{proof}
\phantom{for spacing...}
    \begin{enumerate}[(1)]
        \item  The algorithm is initialized with $R\sqcup A\sqcup B$ as a partition of $V(D)$ and each iteration of the while loop maintains this property.
        \item Label the vertices of $R$ as $v_1,\dots,v_{|R|}$ according to the order they were added to $R$. Suppose $(v_i,v_j)$ is an arc between vertices of $R$. We cannot have $i<j$ since $v_j\in N^{\leq \ell}(v_i)$ is removed from $A$ when $v_i$ is added to $R$, a contradiction. So all arcs $(v_i,v_j)$ in $R$ have $i>j$. Since $R$ has a linear ordering, $R$ is acyclic. 
        \item For every new vertex added to $R$, there are at least $k$ new vertices added to $B$.
        \item The algorithm would not have terminated if there exists $a\in A$ with $|N_A^{\leq \ell}(a)|\geq k$.
        \item Every $b\in B$ belongs to $N^{\leq\ell}(v)$ for some $v\in R$.
        \item Let $a\in A$. By assumption, there exists a directed path $v_0v_1\dots v_{k}=a$ in $D$. No vertices on the path can be in $R$; if say $v_i\in R$, then $a\in N^{\leq \ell}(v_i)$ (since $\ell\geq k$), contradicting that $a\in A$. If the path was completely contained in $A$, then $|N^{\leq \ell}_A(v_0)|\geq k$, contradicting (4). Therefore, some vertex on the path belongs to $B$, so $\dist(B,a)\leq k$. This combined with (5) gives $\dist(R,a)\leq \ell+k$.
        \qedhere
    \end{enumerate}
\end{proof}

\begin{theorem}\label{thm2delta+1}
If $D$ is a digraph with the property that every vertex is the endpoint of some directed path of length $d\geq1$, then $D$ has a $2d$-prekernel of size at most $\frac{1}{d+1}|V(D)|$.
\end{theorem}
\begin{proof}
Use Algorithm~\ref{alg:RAB} with parameters $\ell = k = d$ to obtain sets $R,A$, and $B$.  
By assumption, some vertex $v$ is the start of a path of length $d$. So, $|N^{\le d}(v)| \ge d$, and $R \neq \emptyset$. By Lemma~\ref{lem:ALG}, $R$ is acyclic and every vertex in $A \cup B$ is at distance at most $d+d$ from $R$. 
Therefore, we deduce that $R$ is a $2d$-prekernel. Again by  Lemma~\ref{lem:ALG}, $d|R|\leq |B|$. Hence, $|R|\leq \frac{|B|+|R|}{d+1}\leq \frac{|V(D)|}{d+1}$.
\end{proof}

\begin{cor}\label{cor-to-algorithm}
For $\delta \ge 1$, every digraph $D$ with minimum in-degree $\delta$ contains a $(2\delta+1)$-kernel of size at most $\frac{|V(D)|}{\delta+1}$.
\end{cor}
\begin{proof}
Since $D$ has minimum in-degree $\delta$, every vertex is an endpoint of some directed path of length $\delta$. To see this, just greedily construct such a path in reverse starting with the endpoint. By Theorem~\ref{thm2delta+1}, $D$ has a $2\delta$-prekernel $R$ of size at most $\frac{|V(D)|}{\delta+1}$ which can be made into a $(2\delta+1)$-kernel $Q$ with $|Q|\leq |R|$  by   \Cref{lem:prekernel}. 
\end{proof}

\section{Proof of Theorem~\ref{thm-general q=3} }\label{section-general q=3}

We begin with a warm-up proposition when $\delta =1$. We remark that Spiro's Proposition~2.1 in \cite{spiro2026generalized} is a strengthening of the following as it yields two disjoint $3$-kernels. 
\begin{prop}\label{prop-c13}
    Every digraph $D$ with minimum in-degree $1$ contains a $3$-kernel of size at most $\frac{1}{2}|V(D)|$. This is sharp due to the 2-cycle (or disjoint union of 2-cycles), so $c_{1,3}=\frac{1}{2}$.
\end{prop}
\begin{proof}
    Set $\ell=k=1$, and use Algorithm~\ref{alg:RAB} to obtain a partition $R\sqcup A\sqcup B$ of $V(D)$. By Lemma~\ref{lem:prekernel}, it suffices to show $R$ is a $2$-prekernel with $|R|\leq \frac{1}{2}|V(D)|$. By Lemma~\ref{lem:ALG}, $R$ is acyclic and $|R|\leq |B|\implies |R|\leq \frac{1}{2}|V(D)|$. By Lemma~\ref{lem:ALG} parts (5) and (6), $\dist(R,b)\leq 1$ for all $b\in B$ and $\dist(R,a)\leq 2$ for all $a\in A$.
\end{proof}

For general $\delta$, have the following.
\thmGeneralq
The main consequence is that $c_{\delta,3}\to 0$ as $\delta\to\infty$. 
In contrast, $c_{\delta,2}\geq \frac{1}{2}$ for all $\delta$. We suspect the bound in the theorem is not tight (see Conjecture~\ref{ourConj} in Section~\ref{section-openquestions}).

\begin{proof}
Let $D$ be a digraph with minimum in-degree $\delta\geq1$. We begin by applying Algorithm~\ref{alg:RAB} to $D$ with $\ell = 1$ and $k = \lfloor \sqrt{\delta+1}\rfloor$. Recall this means we iteratively add a vertex $v$ to $R$ if $N_A^{\leq 1}(v)$ contributes at least $k$ new vertices to the set $B$. By the handshaking lemma, there exists a vertex with out-degree at least $\delta\geq k$, so $R$ is nonempty. Define \[
 A_{\text{bad}}=\{a\in A: \dist(B,a)>1\}.
 \]
 If $\delta\in \{1,2\}$, then $k = 1$ and so the set $A_{\text{bad}}$ is empty. Then, as in the proof of Proposition~\ref{prop-c13},  $R$ is a 2-prekernel with $|R|\leq \frac{1}{2}|V(D)|$ and we are done. Otherwise let $P\subseteq A_{bad}$ be a 1-prekernel for $D[A_{\text{bad}}]$ (one always exists; see \cref{remark-1pre}). Observe that $R\cup P$ is a $2$-prekernel for $D$. We have that $R\cup P$ is acyclic since all arcs between $R$ and $P$ must be from $P$ to $R$. We also have $N^{\leq 2}[R\cup P]=V(D)$. Lemma~\ref{lem:prekernel} gives us a $3$-kernel of size at most $|R\cup P|=|R|+|P|$, so it suffices to show $|R|+|P|\leq \frac{|V(D)|}{k+1}$. For $v\in A_{\text{bad}}$, we have $d^{-}_A(v)\geq \delta$. We also know $d^{+}_A(v)\leq k-1$ for every $v\in A$ (otherwise Algorithm \ref{alg:RAB} would have added $v$ to $R$).  Putting these together with the handshaking lemma yields\begin{equation}\label{inoutdeg}
\delta |A_{\text{bad}}| \leq \sum_{v\in A}d^{-}_A(v) = \sum_{v\in A}d^{+}_A(v)\leq (k-1)|A| \implies  \frac{\delta}{k-1}|A_{\text{bad}}|\leq |A|. 
\end{equation}

We have that \begin{align*}
	|V(D)| &= |R| + |B| + |A|\\
	&\geq |R| + k|R| + |A|\quad\text{by Lemma~\ref{lem:ALG}}\\
	&\geq  |R| + k|R| + \frac{\delta}{k-1}|A_{\text{bad}}|\quad\text{by Equation \eqref{inoutdeg}}\\
	&= (k+1)|R| + \frac{\delta}{k-1}|A_{\text{bad}}|.
\end{align*}
Dividing by $(k+1)$ gives $\frac{|V(D)|}{k+1}\geq |R| + \frac{\delta}{k^2-1}|A_{\text{bad}}|$. 
We made the initial choice of $k$ so that $\frac{\delta}{k^2-1} \geq 1$. Then $\frac{|V(D)|}{k+1}\geq |R| + |A_{\text{bad}}| \geq |R|+|P|$. This completes the proof. 

\end{proof}

\section{Proof of Theorem~\ref{thm-1/(1+delta)}}\label{section-1/(1+delta)}

For convenience, we recall Theorem \ref{thm-1/(1+delta)}.
\thmGeneraldelta

The proof strategy is to start with the partition $V(D)=R\sqcup B\sqcup A$ provided by Algorithm~\ref{alg:RAB} (with $k=\ell=\delta)$ . While there is still a vertex $v\in A$ with $\dist(B,v) \geq \lceil\frac{\delta}{2}\rceil + 1$, we will perform the following modification. We find a special vertex  $u_0\in A$ to add to $R$. This vertex $u_0$ will be the first vertex in $A$ on a path that witnesses $\dist(B,v)$. Next, we add $N^{\le \delta}_A(u_0)$ to $B$. We also find a safe set of vertices $A'\subset A$ to set aside. We will ensure that $|A'\cup N^{\le \delta}_A(u_0)|\geq \delta$. This means that each vertex $u_0$ added to $R$ is put in correspondence to some set $A'\cup N^{\le \delta}_A(u_0)$ of size $\geq \delta$. This ensures the final prekernel $R$ satisfies $|R|+\delta|R|\leq |V(D)|$. The following (slightly technical) lemma formalizes this idea.

\begin{lemma}\label{Lemma:keyFordelta}
Let $D$ be a digraph with minimum in-degree $\delta$, and let $L\sqcup K\sqcup M$ be a partition of $V(D)$ satisfying the following conditions. (In Algorithm 2 below, $L$ corresponds to $R\cup B$, $K$ corresponds to $A_{\text{open}}$, and $M$ corresponds to $A_{\text{fixed}}$.)
    \begin{enumerate}[(C1)]
        \item for every vertex $u \in K$, we have $|N_K^{\leq \delta} (v)| < \delta$,
        \item there are no arcs from $M$ to $K$.
    \end{enumerate}
    Suppose further that $\max\{\dist(L,u) : u \in K\} = k+1$ for some $k \geq \lceil \frac{\delta}{2} \rceil$ and $v\in K$ is chosen such that $\dist(L,v) = k+1$. Then there exists a vertex $u_0\in K$ and set $P\subset K$ so that the following hold. 
    \begin{enumerate}[(1)]
        \item $v \in N \coloneqq N_K^{\leq \delta}(u_0)$,
        \item $P \subset K\smallsetminus N$, 
        \item $|N|+|P|\geq \delta$,
        \item every vertex in $P$ has distance at most $\left\lfloor \frac{\delta}{2} \right\rfloor$ from $L$, and
        \item there are no arcs from $P$ to $K\smallsetminus(N\cup P)$.
    \end{enumerate}
\end{lemma}

We first prove the main theorem. The proof of Lemma~\ref{Lemma:keyFordelta} is postponed to the end of the section. 
\begin{proof}[Proof of Theorem \ref{thm-1/(1+delta)}]
If $\delta=1$, then the statement reduces to Proposition~\ref{prop-c13}. So we may assume $D$ is a digraph with minimum in-degree $\delta\geq 2$. Note $\lceil \frac{3}{2}\delta \rceil = \delta + \lceil \frac{\delta}{2} \rceil$.
By Lemma~\ref{lem:prekernel}, it suffices to find a $\left(\delta + \left\lceil\frac{\delta}{2}\right\rceil\right)$-prekernel of size at most $\frac{|V(D)|}{\delta+1}$. We use Algorithm~\ref{alg:RAB} with $\ell = k =\delta$ to obtain a partition $V(D)=R\sqcup A\sqcup B$.
By the handshaking lemma, there exists a vertex $v_0\in V(D)$ with $d^+(v_0)\geq \delta$. Hence $|N^{\le \delta}(v_0)| \ge \delta$ and $R \neq \emptyset$. If $\max\{\dist(B,a): a\in A\} \leq \left\lceil \frac{\delta}{2} \right\rceil $, then $R$ is the desired prekernel. Otherwise we use  Algorithm~\ref{alg:AopenAfixed2} to modify the prekernel $R$. 
\begin{algorithm}
\caption{An algorithm for modification of $R,A,B$}\label{alg:AopenAfixed2}
\begin{algorithmic}
\Require $\delta \geq 2$, $D$ digraph already partitioned by Algorithm~\ref{alg:RAB} to $R\sqcup A\sqcup B$.
\State $A_{\text{open}} \gets A$
\State $A_{\text{fixed}} \gets \emptyset$
\For{$k = \delta,\delta-1,\dots, \left\lceil\frac{\delta}{2}\right\rceil$}
\While{$\exists v \in A_{\text{open}}$ such that $\dist(B,v) = k+1$}
\State Use Lemma~\ref{Lemma:keyFordelta} with $K = A_{\text{open}}$, $L = R\cup B$, $M = A_{\text{fixed}}$ to obtain $u_0$, $N, P$.
\State $R \gets R \cup \{u_0\}$                          
\State $B \gets B \cup N$  
\State $A_{\text{fixed}} \gets A_{\text{fixed}} \cup P$
\State $A_{\text{open}} \gets A_{\text{open}} \smallsetminus (\{u_0\}\cup N\cup P)$
\EndWhile
\EndFor
\end{algorithmic}
\end{algorithm}

Importantly, conditions (C1) and (C2) needed for Lemma~\ref{Lemma:keyFordelta} continue to hold after each pass of the while loop. (C1) continues to hold trivially since $K\coloneqq A_{\text{open}}$ is only getting smaller. Clearly (C2) holds initially since $A_{\text{fixed}}$ is initialized to be empty. Suppose (C2) holds after some number of passes through the while loop, i.e., there are no arcs from $A_{\text{fixed}}'$ to $A_{\text{open}}'$. We show (C2) still holds at the start of the next pass. At the top of the while loop, we will be setting $M\coloneqq A_{\text{fixed}}'\cup P$ and $K\coloneqq A_{\text{open}}'\smallsetminus (\{u_0\}\cup N\cup P)$. By Lemma~\ref{Lemma:keyFordelta} part (5), there are no arcs from $P$ to $A_{\text{open}}'\smallsetminus (N\cup P)$. It follows that there are no arcs from $M$ to $K$, so (C2) holds. This shows the algorithm does not get stuck on the line it needs to apply Lemma~\ref{Lemma:keyFordelta}. The algorithm clearly terminates since $A_{\text{open}}$ is shrinking each while loop pass. After each pass of the while loop, the following can easily be shown to hold. 
\begin{itemize}
        \item $V(D) = A_{\text{open}} \sqcup A_{\text{fixed}} \sqcup B \sqcup R$,
        \item $B$ consists of vertices with distance at most $\delta$ from $R$ (when $N\coloneqq N_K^{\leq \delta}(u_0)$ is added to $B$, $u_0$ is added to $R$),
        \item $R$ is acyclic (arcs only point backwards according to the order each $u_0$ was added to $R$),
        \item $A_{\text{fixed}}$ contains vertices with distance at most $\left\lfloor \frac{\delta}{2} \right\rfloor$ from $B$ (by Lemma~\ref{Lemma:keyFordelta} part (4)),
        \item $\delta |R| \leq |A_{\text{fixed}} \cup B|$ (by Lemma~\ref{Lemma:keyFordelta} part (3)).
    \end{itemize}  
Moreover, once all iterations are finished, 
    \begin{itemize}
        \item $A_{\text{open}}$ contains vertices with distance at most $\left\lceil \frac{\delta}{2} \right\rceil$ from $B$ (since the algorithm has terminated).
    \end{itemize} 
 Therefore, $R$ is a $(\delta+\left\lceil \frac{\delta}{2}\right\rceil )$-prekernel of size at most $\frac{1}{\delta+1}|V(D)|$. This shows $c_{\delta,\left(\delta + \left\lceil\frac{\delta}{2}\right\rceil+1\right)} = \frac{1}{\delta+1}$. The theorem statement (which is for all $q \geq \delta + \left\lceil\frac{\delta}{2}\right\rceil+1$) directly follows by monotonicity of $c_{\delta,q}$.
\end{proof}

We now prove Lemma~\ref{Lemma:keyFordelta}.
\begin{proof}

\begin{figure}[h]
    \centering
    \includegraphics[width=0.75\linewidth]{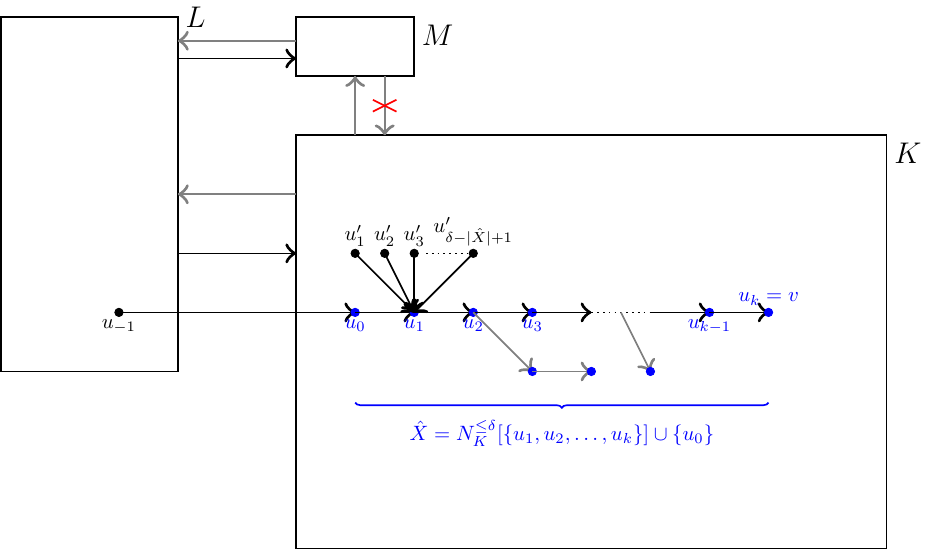}
    \caption{Sets $L$, $K$, and $M$ in  Lemma~\ref{Lemma:keyFordelta}.}
    \label{fig:lkm}
\end{figure}

First, note that since there are no arcs from $M$ to $K$, no path from $L$ to $K$ goes through a vertex in $M$. Next, note that condition (C1) immediately implies there is no path of length $\delta$ in $K$. Since $D$ has minimum in-degree $\delta$, every vertex in $K$ is an endpoint of some path of length $\delta$. Since such a path cannot be completely contained in $K$, we have $\dist(L,u)\leq \delta$ for all $u\in K$. Choose $v\in K$ according to the statement of Lemma~\ref{Lemma:keyFordelta}. Since $\dist(L,v) = k+1$, there exists a path $P = u_0u_1u_2\dots u_k = v$ in $K$ such that for all $0\leq i\leq k$, $\dist(L,u_i) = i+1$. Let $X = N_{K}^{\leq \delta} [\{u_1,\dots, u_k\}]$ and let $\hat{X} = X \cup \{u_0\}$. As $K$ cannot contain any path of length $\delta$, we have $X \subseteq \hat{X} \subseteq N_{K}^{\leq \delta} [u_0]$. Therefore, $k+1 \leq |\hat{X}| \leq \delta$. Before finishing the proof of the theorem, we prove three smaller claims. 

\begin{claim}\label{cl-claim1}
There exist $\delta - (|\hat{X}|-1)$ distinct vertices $u'_1, \dots, u'_{\delta - (|\hat{X}|-1)} \in K\smallsetminus \hat{X}$, which are all in-neighbors of $u_1$.
\end{claim}
\begin{proof}
There are at most $|\hat{X}|-1$ arcs from $\hat{X}$ to $u_1$ since $u_1\in \hat{X}$. Since $u_1$ has in-degree at least $\delta$ and no in-neighbors in $L$ or $M$, there must exist vertices $u'_1, \dots, u'_{\delta - (|\hat{X}|-1)}\in N_K^-(u_1)\smallsetminus \hat{X}$.
\end{proof}

We direct our attention to the vertices $u'_i$. The first observation is that $X \subseteq N_K^{\leq \delta} [u'_i]$ (again since $u_1 \in N_K^{\leq \delta} [u'_i]$ and $K$ cannot contain a path of length $\delta$).
The second observation is that since $u'_i \notin \hat{X}$, there are no arcs from $X$ to $u_i'$. The arc $(u_0,u_i')$ could be present.
\begin{claim}\label{cl:distOfui2}
    For any $i \in \{1,2,\dots, \delta-|\hat{X}|+1\}$, $\dist(L,u'_i) \leq \delta-|X| \leq \delta - k \leq \left\lfloor \frac{\delta}{2}\right\rfloor$.  
\end{claim}
\begin{proof}
Suppose that $\dist(L,u'_i) = b$. Then there exists a path $P' = v_0v_1\dots v_{b}$ witnessing this with $v_0 \in L$, $v_1,\dots, v_{b} \in K$ and $v_b = u_i'$.
Since $u'_i \notin X$, none of the vertices $v_1,\dots,v_{b-1}$ can lie in $X$. We have $\{v_1,\dots,v_b\}\cup X\subseteq N_K^{\leq \delta}[v_1]$, so we deduce $b+|X|\leq N_K^{\leq \delta}[v_1] \leq \delta$, which implies $b \leq \delta - |X|$.    
\end{proof}

\begin{claim}\label{cl:distneighb}
For any $i \in \{1,2,\dots, \delta-|\hat{X}|+1\}$, $|N_K^{\leq \delta} [u'_i]\smallsetminus X|\leq \delta - k$. Moreover, for every $w \in N_K^{\leq \delta}[u'_i]\smallsetminus X$, we have $\dist(L, w) \leq \delta - k \leq \left\lfloor \frac{\delta}{2}\right\rfloor$.
\end{claim}
\begin{proof}
First, the number of vertices in $N_K^{\leq \delta} [u'_i]\smallsetminus X$ is trivially upper bounded by $|N_K^{\leq \delta} [u'_i]| - |X| \leq \delta-k$. Let $w \in N_K^{\leq \delta} [u'_i]\smallsetminus X$ and set $c = \dist(L,w)$. Then certainly if we take the path $v_1\dots u_i'$ from proof of Claim~\ref{cl:distOfui2} and then combine it with the path from $u'_i$ to $w$, this is a walk in $K$ containing $d \geq c$ unique vertices, none of them from $X$ (since otherwise $w \in X$).
Since the concatenated paths and $X$ are both contained in $N_K^{\leq \delta}[v_1]$, we obtain that 
$d+|X| \leq N_K^{\leq \delta}[v_1] \leq \delta$. Hence $c \leq d \leq \delta - |X| \leq \delta - k$. 
\end{proof}

We are now ready to complete the proof by verifying (1)-(5) hold.  Define the sets \[
N\coloneqq N_K^{\leq \delta}(u_0) \quad\text{and}\quad  P \coloneqq N_K^{\leq \delta}[u'_1, \dots, u'_{\delta - (|\hat{X}|-1)}] \smallsetminus N.\]
Note (1) and (2) follow immediately. By Claim~\ref{cl-claim1}, the set $P$ contains $\delta-(|\hat{X}|-1)$ vertices distinct from $\hat{X}$. Since $X\subseteq N$, we have $|N|+|P|\geq \left(|\hat{X}|-1 \right)+ \left(\delta-(|\hat{X}|-1)\right)=\delta$ which gives (3). The second part of Claim~\ref{cl:distneighb} shows that $\dist(L,w) \leq \left\lfloor \frac{\delta}{2}\right\rfloor$ for all $w\in P$ giving (4). It follows from the first part of Claim~\ref{cl:distneighb} that $N_K^{\leq \delta}[u'_1]\smallsetminus X = N_K^{\leq \delta-k}[u'_1]\smallsetminus X$  for each $u_i'$. If there were an arc $(x,y)$ from $P$ to $K\smallsetminus(N\cup P)$, then $y\notin P\cup N\implies \dist_K(\{u'_1, \dots, u'_{\delta - (|\hat{X}|-1)}\}, y)=\delta+1$, a contradiction. This establishes (5) and completes the proof. 
\end{proof}

\section{Open questions}\label{section-openquestions}
\subsection{Sharpening the bounds on $c_{\delta,q}$}
The problem of determining precise bounds on $c_{\delta,q}$ is still open for many ranges of $\delta$ and $q$. To briefly recap, we have shown that 
\begin{itemize}
    \item (Theorem~\ref{thm-mono}) The constants $c_{\delta,q}$ are weakly decreasing in $\delta$ and in $q$;
    \item (Theorem~\ref{thm-general q=3}) for all $q \geq 3$ and $\delta \geq 1$,
    \[\frac{1}{\delta+1} \leq c_{\delta,q}\leq \frac{1}{\lfloor\sqrt{\delta+1}\rfloor+1};\]
    \item (Theorem~\ref{thm-1/(1+delta)}) for any $\delta \geq 1$ and $q \geq \left\lceil\frac{3\delta}{2}\right\rceil + 1$, 
    \[\frac{1}{\delta+1} = c_{\delta,q}.\]
\end{itemize}

Recall that the lower bound $\frac{1}{\delta+1}\leq c_{\delta,q}$ follows from the rather trivial Example~\ref{ex-deltaplusone}. We suspect the lower bound is the truth and conjecture the following.
\begin{conj}\label{ourConj}
    For all $\delta\geq 1$ and $q\geq 3$,  $c_{\delta,q}=\frac{1}{\delta+1}$.
\end{conj}
By monotonicity, this conjecture is equivalent to the assertion that $c_{\delta,3}=\frac{1}{\delta+1}$. Since we know $c_{1,3}=\frac{1}{2}$ (Proposition~\ref{prop-c13}), the first place to check for counterexamples is the $\delta=2$ case. A counterexample to the conjecture would show the asymptotic behavior of $c_{\delta,q}$ has the potential to be complex.

\subsection{Variants for $(k,\ell)$-kernels}

In the early 1980s, Kwa{\'s}nik (\cite{kwasnik1980k} \cite{kwasnik1981generalization}) defined a more general object, $(k,\ell)$-kernels. A set $K$ is called a $(k,\ell)$-kernel if it is both \textit{$k$-independent} ($u,v\in K,\, u\neq v\implies \dist(u,v)\geq k$) and \textit{$\ell$-dominating} ($N^{\leq \ell}[K]=V(D)$). Note that the $q$-kernels in our paper are precisely $(2,q)$-kernels. Most of the work on $(k,\ell)$-kernels has been on proving existence results (see \cite{galeana1990existence,galeana2014existence,galeana1998semikernels,ramoul2017new}), as opposed to bounding their size. A future line of work could be determining how minimum in-degree impacts the existence and size of $(k,\ell)$-kernels. To this end, define the constant $c_{\delta, (k,\ell)}$ to be the smallest constant such that every digraph $D$ with minimum in-degree $\delta$ has a $(k,\ell)$-kernel of size at most $c_{\delta, (k,\ell)}|V(D)|$. Leave $c_{\delta, (k,\ell)}$ undefined if there exists a minimum in-degree $\delta$ digraph with no $(k,\ell)$-kernel. Galeana-S\'anchez and Li \cite{galeana1998semikernels} showed that every digraph contains a $(k,2k-2)$-kernel. This is sharp since a directed $C_{2k-1}$ has no $(k,2k-3)$-kernel. First, one might ask if such examples persist if we require large minimum in-degree.
\begin{question}
Let $k\geq 2$. Does there exists a large enough $\delta$ such that all digraphs with minimum in-degree $\delta$ contain a $(k,2k-3)$-kernel? 
\end{question}
For example, the answer is no when $k=2$ (consider any tournament with minimum in-degree $\delta$). After investigating the existence question, one can ask how small $(k,\ell)$-kernels can be made for the regimes when $(k,\ell)$-kernels exist. Since $(k,2k-2)$-kernels always exist, we can ask the following.
\begin{question}
    For $\delta\geq 1$ and $k\geq 2$, estimate the asymptotics of the constants $c_{\delta,(k,2k-2)}$. In particular, is it true that for any $k\geq 2$, $c_{\delta,(k,2k-2)}\to 0$ as $\delta\to\infty$? 
\end{question}

\section*{Acknowledgments}
This project began during the 2025 Graduate Research Workshop in Combinatorics. The workshop was supported by Iowa State University, the Combinatorics Foundation, and the U.S. National
Science Foundation (grant award number 1953445).

The third author was supported by the grant 23-06815M of the Grant Agency of the Czech Republic.

\bibliographystyle{plain}
\bibliography{finalbib} 

\end{document}